\documentclass[12pt,reqno]{article}

\usepackage[usenames]{color}
\usepackage{amssymb}
\usepackage{graphicx}
\usepackage{amscd}

\usepackage{amsthm}
\newtheorem{theorem}{Theorem}

\newtheorem{proposition}[theorem]{Proposition}
\newtheorem{corollary}[theorem]{Corollary}

\theoremstyle{definition}

\newtheorem{example}[theorem]{Example}

\usepackage[colorlinks=true,
linkcolor=webgreen, filecolor=webbrown,
citecolor=webgreen]{hyperref}

\definecolor{webgreen}{rgb}{0,.5,0}
\definecolor{webbrown}{rgb}{.6,0,0}

\usepackage{color}

\usepackage{float}

\usepackage{graphics,amsmath,amssymb}
\usepackage{amsfonts}
\usepackage{latexsym}
\usepackage{epsf}

\newcommand{\seqnum}[1]{\href{http://www.research.att.com/cgi-bin/access.cgi/as/~njas/sequences/eisA.cgi?Anum=#1}{\underline{#1}}}

\begin{document}

\begin{center}
\vskip 1cm{\LARGE\bf Chebyshev moments and Riordan involutions} \vskip 1cm \large
Paul Barry\\
School of Science\\
Waterford Institute of Technology\\
Ireland\\
\href{mailto:pbarry@wit.ie}{\tt pbarry@wit.ie}
\end{center}
\vskip .2 in

\begin{abstract} We show that the coefficient array of a family of Chebyshev moments defines an involution in the group of Riordan arrays. We then extend this result to certain families of $d$-orthogonal polynomials.\end{abstract}

\section{Introduction} In this note, we exhibit two families of parameterized generalized Chebyshev polynomials \cite{Chebyshev} whose moments have coefficient arrays that are involutions in the Riordan group of lower-triangular matrices. The structure of the paper is as follows.
\begin{itemize}
\item This introduction
\item Preliminaries on Riordan arrays
\item Riordan involutions
\item Riordan arrays and orthogonal polynomials
\item The main results
\item A general result
\item A factorization theorem
\item Further results
\item Conclusions
\item Appendix - a look at the case of the Chebyshev polynomials of the first kind
\end{itemize}
Readers who are familiar with any of the above introductory topics may safely skip the relevant section.

\section{Preliminaries on Riordan arrays}
We recall some facts about Riordan arrays in this introductory section.
A Riordan array \cite{Book, SGWW} is defined by a pair of power series
$$g(x)=g_0 + g_1 x + g_2 x^2 + \cdots=\sum_{n=0}^{\infty} g_n x^n,$$ and
$$f(x)=f_1 x + f_2 x^2+ f_3 x^3 + \cdots = \sum_{n=1}^{\infty} f_n x^n.$$
We require that $g_0 \ne 0$ (and hence $g(x)$ is invertible, with inverse $\frac{1}{g(x)}$), while we also require that
$f_0=0$ and $f_1 \ne 0$ (hence $f(x)$ has a compositional inverse $\bar{f}(x)=\text{Rev}(f)(x)$ defined by $f(\bar{f}(x))=x$).
The set of such pairs $(g(x), f(x))$ forms a group (called the Riordan group \cite{SGWW}) with multiplication
$$(g(x), f(x)) \cdot (u(x), v(x))=(g(x)u(f(x)), v(f(x)),$$ and with inverses given by
$$(g(x), f(x))^{-1}=\left(\frac{1}{g(\bar{f}(x))}, \bar{f}(x)\right).$$
The coefficients of the power series may be drawn from any ring (for example, the integers $\mathbb{Z}$) where these operations make sense. To each such ring there exists a corresponding Riordan group.

There is a matrix representation of this group, where, to the element $(g(x), f(x))$, we associate the matrix
$\left(a_{n,k}\right)_{0 \le n,k \le \infty}$ with general element
$$t_{n,k}=[x^n] g(x)f(x)^k.$$
Here, $[x^n]$ is the functional that extracts the coefficient of $x^n$ in a power series \cite{MC}. In this representation, the group law corresponds to ordinary matrix multiplication, and the inverse of $(g(x), f(x))$ is represented by the inverse of $\left(t_{n,k}\right)_{0 \le n,k \le \infty}$.

The Fundamental Theorem of Riordan arrays is the rule
$$(g(x), f(x))\cdot h(x)=g(x)h(f(x)),$$ detailing how an array $(g(x), f(x))$ can act on a power series. This corresponds to the matrix $(t_{n,k})$ multiplying the column vector $(h_0, h_1, h_2,\ldots)^T$.

\begin{example} Pascal's triangle, also known as the binomial matrix, is defined by the Riordan group element
$$B=\left(\frac{1}{1-x}, \frac{x}{1-x}\right).$$ This means that we have
$$\binom{n}{k}=[x^n] \frac{1}{1-x} \left(\frac{x}{1-x}\right)^k.$$
To see that this is so, we need to be familiar with the rules of operation of the functional $[x^n]$ \cite{MC}.
We have
\begin{align*}
[x^n] \frac{1}{1-x} \left(\frac{x}{1-x}\right)^k&=[x^n] \frac{x^k}{(1-x)^{k+1}}\\
&= [x^{n-k}] (1-x)^{-(k+1)}\\
&= [x^{n-k}] \sum_{j=0}^{\infty} \binom{-(k+1)}{j}(-1)^j x^j\\
&= [x^{n-k}] \sum_{j=0}^{\infty} \binom{k+1+j-1}{j}x^j\\
&= [x^{n-k}] \sum_{j=0}^{\infty} \binom{k+j}{j} x^j \\
&= \binom{k+n-k}{n-k}=\binom{n}{n-k}=\binom{n}{k}.\end{align*}
\end{example}

The binomial matrix is an element of the Bell subgroup of the Riordan group, consisting of arrays of the form $(g(x), xg(x))$. It is also an element of the hitting time subgroup, which consists of arrays of the form
$\left(\frac{x f'(x)}{f(x)}, f(x)\right)$. Arrays of the form $(1, f(x))$ belong to the associated or Lagrange subgroup of the Riordan group.

Note that all the arrays in this note are lower triangular matrices of infinite extent. We show appropriate truncations.

Many examples of sequences and  Riordan arrays are documented in the On-Line Encyclopedia of Integer Sequences (OEIS) \cite{SL1, SL2}. Sequences are frequently referred to by their
OEIS number. For instance, the binomial matrix $\mathbf{B}=\left(\frac{1}{1-x}, \frac{x}{1-x}\right)$ (``Pascal's triangle'') is \seqnum{A007318}. The diagonal sums of this matrix are the Fibonacci numbers $F_{n+1}$. The sequence $F_n$ is \seqnum{A000045}. In the sequel we will not distinguish between an array pair $(g(x), f(x))$ and its matrix representation.

The Hankel transform of a sequence $a_n$ is the sequence of determinants $h_n=|a_{i+j}|_{i \le i,j \le n}$ \cite{Kratt, Layman}.

The Catalan numbers $C_n=\frac{1}{n+1} \binom{2n}{n}$ \seqnum{A000108} have generating function
$$c(x)=\frac{1-\sqrt{1-4x}}{2x}.$$
We note that $$\text{Rev}(xc(x)) = x(1-x).$$ The Catalan numbers $C_n$ are the unique numbers such that the sequences $C_n$ and $C_{n+1}$ both have their Hankel transforms given by $h_n=1$ for all $n \ge 0$.

There are a number of combinatorially important Riordan arrays that are closely related to the Catalan numbers. The principal ones are $(1, xc(x))$ \seqnum{A106566}, $(1, c(x)-1)=(1,xc(x)^2)$ \seqnum{A128899}, $(c(x), xc(x))$ \seqnum{A033184} and $(c(x)^2, xc(x)^2)$ \seqnum{A039598}. These, and their reverse triangles, are collectively known as Catalan matrices. For instance, the matrix $(1, xc(x))$ begins
$$\left(
\begin{array}{ccccccc}
 1 & 0 & 0 & 0 & 0 & 0 & 0 \\
 0 & 1 & 0 & 0 & 0 & 0 & 0 \\
 0 & 1 & 1 & 0 & 0 & 0 & 0 \\
 0 & 2 & 2 & 1 & 0 & 0 & 0 \\
 0 & 5 & 5 & 3 & 1 & 0 & 0 \\
 0 & 14 & 14 & 9 & 4 & 1 & 0 \\
 0 & 42 & 42 & 28 & 14 & 5 & 1 \\
\end{array}
\right).$$
This is \seqnum{A106566}.

\section{Riordan involutions}
An involution in the group of Riordan arrays is a Riordan array $A$ such that $A^2=I$, that is, $A$ is an element of order $2$ \cite{Cameron, Cheon, Phulara}. For instance, the arrays $I=(1,x)$ and $(1,-x)$ are involutions. The result is trivial for $I$, and for $(1,-x)$ we have
$$(1, -x) \cdot (1, -x)= (1, -x(-x))=(1, -(-x))=(1, x).$$
A less trivial example is the Riordan array $\left(\frac{1}{1-x}, \frac{-x}{1-x}\right)$ which is a signed version of Pascal's triangle (the binomial matrix) which begins
$$\left(
\begin{array}{ccccccc}
 1 & 0 & 0 & 0 & 0 & 0 & 0 \\
 1 & -1 & 0 & 0 & 0 & 0 & 0 \\
 1 & -2 & 1 & 0 & 0 & 0 & 0 \\
 1 & -3 & 3 & -1 & 0 & 0 & 0 \\
 1 & -4 & 6 & -4 & 1 & 0 & 0 \\
 1 & -5 & 10 & -10 & 5 & -1 & 0 \\
 1 & -6 & 15 & -20 & 15 & -6 & 1 \\
\end{array}
\right).$$
We have
\begin{align*}\left(\frac{1}{1-x}, \frac{-x}{1-x}\right)\cdot \left(\frac{1}{1-x}, \frac{-x}{1-x}\right)
&=\left(\frac{1}{1-x} \frac{1}{1-\frac{-x}{1-x}}, \frac{-\frac{-x}{1-x}}{1-\frac{-x}{1-x}}\right)\\
&=\left(\frac{1}{1-x} (1-x), \frac{\frac{x}{1-x}}{\frac{1}{1-x}}\right)\\
&=(1,x).\end{align*}
Thus the Riordan array $\left(\frac{1}{1-x}, \frac{-x}{1-x}\right)$ is a Riordan involution. The next result follows from a general principal to be found in \cite{Phulara}. Here we give a direct proof.
\begin{proposition} The Riordan array $(c(x), -xc(x)^3)$ is a Riordan involution.
\end{proposition}
\begin{proof} We must show that $(c(x), -xc(x)^3)^{-1}=(c(x), -xc(x)^3)$. Thus we calculate $(c(x), -xc(x)^3)^{-1}$.
We have
$$-x c(x)^3=-\frac{1-3x+(x-1)\sqrt{1-4x}}{2x^2}.$$ Thus we must solve the equation
$$-\frac{1-3u+(u-1)\sqrt{1-4u}}{2u^2}=x$$ to solve for the reversion of $-xc(x)^3$.
Equivalently, we must solve
$$1-4u=\left(\frac{-2xu^2-1+3u}{u-1}\right)^2$$ for $u=u(x)$.
The solution of this equation that satisfies $u(0)=0$ is given by
$$u(x)=-\frac{1-3x+(x-1)\sqrt{1-4x}}{2x^2}=-xc(x)^3.$$
We wish now to show that
$$\frac{1}{c(-xc(x)^3)}=c(x).$$ In order to verify this, we substitute $x(1-x)=\overline{xc(x)}$ for $x$.
We find that
$$\frac{1}{c(-x(1-x)c(x(1-x))^3)}=\frac{1}{1-x}.$$
Similarly, $$c(x(1-x))=\frac{1}{1-x}.$$ By the uniqueness of the reversion process, we conclude that
$$\frac{1}{c(-xc(x)^3)}=c(x).$$
Thus we have $(c(x), -xc(x)^3)^{-1}=(c(x), -xc(x)^3)$ and hence $\left(c(x), -xc(x)^3\right)$ is a Riordan involution.
\end{proof}
The Riordan array $\left(c(x), -xc(x)^3\right)$ begins
$$\left(
\begin{array}{ccccccc}
 1 & 0 & 0 & 0 & 0 & 0 & 0 \\
 1 & -1 & 0 & 0 & 0 & 0 & 0 \\
 2 & -4 & 1 & 0 & 0 & 0 & 0 \\
 5 & -14 & 7 & -1 & 0 & 0 & 0 \\
 14 & -48 & 35 & -10 & 1 & 0 & 0 \\
 42 & -165 & 154 & -65 & 13 & -1 & 0 \\
 132 & -572 & 637 & -350 & 104 & -16 & 1 \\
\end{array}
\right).$$ The general term of the unsigned matrix is $\frac{3k+1}{n+2k+1}\binom{2n+k}{n-k}$. We note that the diagonal sums of this unsigned number triangle $(c(x), xc(x)^3)$ are the moment sequence $a_n$ that begins
$$1, 1, 3, 9, 29, 97, 333, 1165, 4135, 14845, \ldots,$$ or \seqnum{A081696}.
We have the integral representation \cite{Classical}
$$a_n =\frac{1}{\pi} \int_0^4 \left(\frac{-\sqrt{4-x}}{x^2-4x-1}\right)\,dx+\frac{1}{\sqrt{5}}(2-\sqrt{5})^n.$$ The Hankel transform of $a_n$ is $2^n$ and that of the once shifted sequence $a_{n+1}$ is $2^n(1-n)$.
\begin{corollary} The Riordan array $\left(c(x)^2, -xc(x)^3\right)$ is a Riordan involution.
\end{corollary}
\begin{proof} This follows by the general result that if $(g(x), f(x))$ is an involution then so is $(g(x)^n, f(x))$ for all $n \in \mathbb{Z}$.
\end{proof}
The Riordan array $\left(c(x)^2, -xc(x)^3\right)$ begins
$$\left(
\begin{array}{ccccccc}
 1 & 0 & 0 & 0 & 0 & 0 & 0 \\
 2 & -1 & 0 & 0 & 0 & 0 & 0 \\
 5 & -5 & 1 & 0 & 0 & 0 & 0 \\
 14 & -20 & 8 & -1 & 0 & 0 & 0 \\
 42 & -75 & 44 & -11 & 1 & 0 & 0 \\
 132 & -275 & 208 & -77 & 14 & -1 & 0 \\
 429 & -1001 & 910 & -440 & 119 & -17 & 1 \\
\end{array}
\right).$$
The unsigned matrix is \seqnum{A107842}, with general term $\frac{3k+2}{n+2k+2} \binom{2n+k+1}{n-k}$. The diagonal sums of the unsigned matrix $(c(x)^2, xc(x)^3)$ begin
$$1, 2, 6, 19, 63, 215, 749, 2650, 9490, 34318, 125104, 459152, \ldots.$$ This is \seqnum{A109262}. The Hankel transform of this sequence is given by
$$h_n = F_{2n+1}.$$
We note further that these two diagonal sum sequences appear respectively as the first and second column in the inverse of the Riordan array $(1-x-x^2, x(1-x))$ \seqnum{A109267}.

\section{Riordan arrays and orthogonal polynomials}
We recall the following well-known results (the first is known as ``Favard's Theorem''), which we essentially reproduce from
\cite{Kratt}, to specify the links between orthogonal polynomials \cite{Chihara, Gautschi, Szego}, three term recurrences, and the recurrence coefficients and the generating function of
the moment sequence of the orthogonal polynomials.
\begin{theorem} \label{ThreeT}\cite{Kratt} (Cf. \cite{Viennot}, Th\'eor\`eme $9$ on p.I-4, or \cite{Wall}, Theorem $50.1$). Let $(p_n(x))_{n\ge 0}$
be a sequence of monic polynomials, the polynomial $p_n(x)$ having degree $n=0,1,\ldots$ Then the sequence $(p_n(x))$ is (formally)
orthogonal if and only if there exist sequences $(\alpha_n)_{n\ge 0}$ and $(\beta_n)_{n\ge 1}$ with $\beta_n \neq 0$ for all $n\ge 1$,
such that the three-term recurrence
$$p_{n+1}=(x-\alpha_n)p_n(x)-\beta_n p_{n-1}(x), \quad \text{for}\quad n\ge 1, $$
holds, with initial conditions $p_0(x)=1$ and $p_1(x)=x-\alpha_0$.
\end{theorem}

\begin{theorem} \label{CF} \cite{Kratt} (Cf. \cite{Viennot}, Proposition 1, (7), on p. V-$5$, or \cite{Wall}, Theorem $51.1$). Let $(p_n(x))_{n\ge
0}$ be a sequence of monic polynomials, which is orthogonal with respect to some functional $\mathcal{L}$. Let
$$p_{n+1}=(x-\alpha_n)p_n(x)-\beta_n p_{n-1}(x), \quad \text{for}\quad n\ge 1, $$ be the corresponding three-term recurrence which is
guaranteed by Favard's theorem. Then the generating function
$$g(x)=\sum_{k=0}^{\infty} \mu_k x^k $$ for the moments $\mu_k=\mathcal{L}(x^k)$ satisfies
$$g(x)=\cfrac{\mu_0}{1-\alpha_0 x-
\cfrac{\beta_1 x^2}{1-\alpha_1 x -
\cfrac{\beta_2 x^2}{1-\alpha_2 x -
\cfrac{\beta_3 x^2}{1-\alpha_3 x -\cdots}}}}.$$
\end{theorem}

The {\it
Hankel
transform} of a given sequence
$A=\{a_0,a_1,a_2,...\}$ is the
sequence of Hankel determinants $\{h_0, h_1, h_2,\dots \}$
where
$h_{n}=|a_{i+j}|_{i,j=0}^{n}$, i.e

\begin{center} \begin{equation}
 \label{gen1}
 A=\{a_n\}_{n\in\mathbb N_0}\quad \rightarrow \quad
 h=\{h_n\}_{n\in\mathbb N_0}:\quad
h_n=\left| \begin{array}{ccccc}
 a_0\ & a_1\  & \cdots & a_n  &  \\
 a_1\ & a_2\  &        & a_{n+1}  \\
\vdots &      & \ddots &          \\
 a_n\ & a_{n+1}\ &    & a_{2n}
\end{array} \right|. \end{equation} \end{center} The Hankel
transform of a sequence $a_n$ and its binomial transform are
equal.

In the case that $a_n$ has a generating function $g(x)$ expressible in the form
$$g(x)=\cfrac{a_0}{1-\alpha_0 x-
\cfrac{\beta_1 x^2}{1-\alpha_1 x-
\cfrac{\beta_2 x^2}{1-\alpha_2 x-
\cfrac{\beta_3 x^2}{1-\alpha_3 x-\cdots}}}}$$ then
we have \cite{Kratt} Heilermann's formula
\begin{equation}\label{Kratt} h_n = a_0^{n+1} \beta_1^n\beta_2^{n-1}\cdots \beta_{n-1}^2\beta_n=a_0^{n+1}\prod_{k=1}^n
\beta_k^{n+1-k}.\end{equation}
Note that this is independent of $\alpha_n$.

Along with the two power series $g(x)$ and $f(x)$, we can define two associated power series
$$A(x)=\frac{x}{\bar{f}(x)},$$ and
$$Z(x)=\frac{1}{\bar{f}(x)}\left(1-\frac{1}{g(\bar{f}(x))}\right).$$
Note that the notation $\bar{f}(x)$ denotes the compositional inverse of the power series $f$. Thus we have
$$\bar{f}(f(x))=x, \quad\text{and}\quad f(\bar{f}(x))=x.$$
We shall also use the notation $\bar{f}(x)=\text{Rev}(f)(x)$.

The (infinite) matrix whose bivariate generating function is given by
$$Z(x)+\frac{A(x)y}{1-xy}$$
is called the production matrix of $M$ \cite{ProdMat_0, ProdMat, PW}. It is equal to the matrix
$$M^{-1}\bar{M}$$ where $\bar{M}$ is the matrix $M$ with its first row removed. It is an infinite lower Hessenberg matrix.

If the Riordan array $M$ has a production matrix $P$ defined by $A(x)$ and $Z(x)$, then we can show that
$$M^{-1}=\left(1-x \frac{Z(x)}{A(x)}, \frac{x}{A(x)}\right).$$

If $Z(x)=\gamma  + \delta x$, and $A(x)=1+\alpha x + \beta x^2$ (where we assume that $\delta \ne 0$ and $\beta \ne 0$), then $P$ will be tri-diagonal. The matrix $P$ begins  $$\left(
\begin{array}{cccccccc}
\gamma & 1 & 0 & 0 & 0 & 0 & 0 & 0 \\
\delta & \alpha & 1 & 0 & 0 & 0 & 0 & 0 \\
 0 & \beta & \alpha & 1 & 0 & 0 & 0 & 0 \\
 0 & 0 & \beta & \alpha & 1 & 0 & 0 & 0 \\
 0 & 0 & 0 & \beta & \alpha & 1 & 0 & 0 \\
 0 & 0 & 0 & 0 & \beta & \alpha & 1 & 0 \\
 0 & 0 & 0 & 0 & 0 & \beta & \alpha & 1 \\
 0 & 0 & 0 & 0 & 0 & 0 & \beta & \alpha \\
\end{array}
\right).$$
In this case, we have
$$M^{-1}=\left(1-x \frac{\gamma+\delta x}{1+\alpha x + \beta x^2}, \frac{x}{1+\alpha x + \beta x^2}\right)=
\left(\frac{1+(\alpha-\gamma)x+(\beta-\delta)x^2}{1+\alpha x+\beta x^2},\frac{x}{1+\alpha x + \beta x^2}\right),$$
and $R=M^{-1}$ will be the coefficient array of a family of orthogonal polynomials $P_n(x)$ \cite{Meixner} where
$$P_n(x)=\sum_{k=0}^n p_{n,k} x^k,$$ where
the general $(n,k)$-th term of $R$ is $p_{n,k}.$ In this case, we call $M$ the \emph{moment matrix} of the family of orthogonal polynomials $P_n(x)$ \cite{Barry_moments}. More generally, if the matrix $M$ has a production matrix that has a $k$-diagonal form, we will continue to call the elements of the first column of $M$ ``moments''.

\begin{example} The Riordan array $\left(\frac{1}{1+x^2}, \frac{x}{1+x^2}\right)$ begins
$$\left(
\begin{array}{ccccccc}
 1 & 0 & 0 & 0 & 0 & 0 & 0 \\
 0 & 1 & 0 & 0 & 0 & 0 & 0 \\
 -1 & 0 & 1 & 0 & 0 & 0 & 0 \\
 0 & -2 & 0 & 1 & 0 & 0 & 0 \\
 1 & 0 & -3 & 0 & 1 & 0 & 0 \\
 0 & 3 & 0 & -4 & 0 & 1 & 0 \\
 -1 & 0 & 6 & 0 & -5 & 0 & 1 \\
\end{array}
\right),$$ and defines the polynomial sequence $P_n(x)$ that begins
$$1,x,x^2-1,x^3-2 x,x^4-3 x^2+1,x^5-4 x^3+3 x,x^6-5 x^4+6 x^2-1,\ldots.$$
These are the scaled Chebyshev polynomials of the second kind $U_n(x/2)$ where we have
$$U_n(x)=[x^n]\frac{1}{1-2xt+t^2},$$ or
$$U_n(x)=\sum_{k=0}^{\lfloor \frac{n}{2} \rfloor} \binom{n-k}{k}(-1)^k (2x)^{n-2k}.$$
\end{example}

More generally, we have the following result \cite{Meixner}.
\begin{proposition}
The Riordan array $\left(\frac{1-\lambda x - \mu x^2}{1+rx+sx^2},\frac{x}{1+rx+sx^2}\right)$ is the coefficient array of the generalized Chebyshev polynomials of the second kind given by
$$Q_n(x)=(\sqrt{s})^n U_n\left(\frac{x-r}{2\sqrt{s}}\right)-\lambda(\sqrt{s})^{n-1} U_{n-1}\left(\frac{x-r}{2\sqrt{s}}\right)-\mu(\sqrt{s})^{n-2} U_{n-2}\left(\frac{x-r}{2\sqrt{s}}\right), \quad n=0,1,2,\ldots$$
\end{proposition}
\begin{proof}
We have $$U_n(x)=[x^n]\frac{1}{1-2xt+t^2}.$$ By the method of coefficients \cite{MC} we then have
$$ [x^n]\frac{t}{1-2xt+t^2}=[x^{n-1}]\frac{1}{1-2xt+t^2}=U_{n-1}(x)$$ and similarly
$$ [x^n]\frac{t^2}{1-2xt+t^2}=[x^{n-2}]\frac{1}{1-2xt+t^2}=U_{n-2}(x).$$
\end{proof}

\section{The main results}

\begin{theorem} The coefficient array of the moments $\mu_n^{(1)}(y)$ of the generalized Chebyshev orthogonal polynomials defined by
$$\left(\frac{1+yx+yx^2}{1+x}, \frac{x}{(1+x)^2}\right)$$ is the Riordan involution
$$(c(x), -xc(x)^3).$$
The coefficient array of the moments $\mu_n^{(2)}(y)$ of the generalized Chebyshev orthogonal polynomials defined by
$$\left(\frac{1+yx+yx^2}{(1+x)^2}, \frac{x}{(1+x)^2}\right)$$ is the Riordan involution
$$(c(x)^2, -xc(x)^3).$$
\end{theorem}
\begin{proof}
The moments of the orthogonal polynomials whose coefficient array is given by $\left(\frac{1+yx+yx^2}{1+x}, \frac{x}{(1+x)^2}\right)$ are given by the elements of the first column of the inverse matrix
$$\left(\frac{1+yx+yx^2}{1+x}, \frac{x}{(1+x)^2}\right)^{-1}.$$
Letting $g_y^{(1)}(x)=\frac{1+yx+yx^2}{1+x}$, we have
$$\left(\frac{1+yx+yx^2}{1+x}, \frac{x}{(1+x)^2}\right)^{-1}=\left(\frac{1}{g_y^{(1)}(x c(x)^2)}, x c(x)^2\right).$$
Thus the coefficient array of the moments has generating function given by $\frac{1}{g_y^{(1)}(x c(x)^2)}$.
We calculate that
$$\frac{1}{g_y^{(1)}(x c(x)^2)}=\frac{y+(1-2y)x-(x-y)\sqrt{1-4x}}{2(x^2+xy(y-3)+y)}.$$
On the other hand, we have that the bivariate generating function of the Riordan involution $(c(x), -xc(x)^3)$ is given by $\frac{c(x)}{1+yxc(x)^3}$. The result now follows since
$$\frac{c(x)}{1+yxc(x)^3}=\frac{y+(1-2y)x-(x-y)\sqrt{1-4x}}{2(x^2+xy(y-3)+y)}.$$
The moments of the orthogonal polynomials whose coefficient array is given by $\left(\frac{1+yx+yx^2}{(1+x)^2}, \frac{x}{(1+x)^2}\right)$ are given by the elements of the first column of the inverse matrix
$$\left(\frac{1+yx+yx^2}{(1+x)^2}, \frac{x}{(1+x)^2}\right)^{-1}.$$
Letting $g_y^{(2)}(x)=\frac{1+yx+yx^2}{(1+x)^2}$, we have
$$\left(\frac{1+yx+yx^2}{(1+x)^2}, \frac{x}{(1+x)^2}\right)^{-1}=\left(\frac{1}{g_y^{(2)}(x c(x)^2)}, x c(x)^2\right).$$
Thus the coefficient array of the moments has generating function given by $\frac{1}{g_y^{(2)}(x c(x)^2)}$.
We calculate that
$$\frac{1}{g_y^{(2)}(x c(x)^2)}=\frac{1+y-2x-(1-y)\sqrt{1-4x}}{2(x^2+xy(y-3)+y)}.$$
On the other hand, we have that the bivariate generating function of the Riordan involution $(c(x)^2, -xc(x)^3)$ is given by $\frac{c(x)^2}{1+yxc(x)^3}$. The result now follows since
$$\frac{c(x)^2}{1+yxc(x)^3}=\frac{1+y-2x-(1-y)\sqrt{1-4x}}{2(x^2+xy(y-3)+y)}.$$
\end{proof}

The generalized Chebyshev orthogonal polynomials defined by the array $\left(\frac{1+yx+yx^2}{(1+x)^2}, \frac{x}{(1+x)^2}\right)$ are given by
$$P_n(x)=U_n\left(\frac{x-2}{2}\right)+yU_{n-1}\left(\frac{x-2}{2}\right)+yU_{n-2}\left(\frac{x-2}{2}\right).$$
We may ask whether there are values of $r$ and $s$ such that $\left(\frac{1+ryx+syx^2}{(1+x)^2}, \frac{x}{(1+x)^2}\right)$ leads to a Riordan involution, other than $r=s=1$. The answer is no. First, we require that $r=1$ to ensure that the diagonal of the square of the moment coefficient array be all $1$s. Secondly, if we take the case of $\left(\frac{1+yx+syx^2}{(1+x)^2}, \frac{x}{(1+x)^2}\right)$, we find that the square of the moment coefficient matrix begins
$$\left(
\begin{array}{cccccc}
 1 & 0 & 0 & 0 & 0 & 0 \\
 0 & 1 & 0 & 0 & 0 & 0 \\
 2-2 s & 0 & 1 & 0 & 0 & 0 \\
 2 (1-s) & -2 \left(s^2+s-2\right) & 0 & 1 & 0 & 0 \\
 5 s^2-16 s+11 & -s^3-2 s^2-s+4 & -4 s^2-2 s+6 & 0 & 1 & 0 \\
 8 \left(s^2-4 s+3\right) & 8 s^3-14 s^2-20 s+26 & -2 \left(s^3+2 s^2-3\right) & -6 s^2-2 s+8 & 0 & 1 \\
\end{array}
\right).$$
Thus $s=1$ is a necessary condition for there to be an involution. By the above, it is also a sufficient condition.
\section{The Hankel transform of the moments}
We have the following result.
\begin{proposition} The generating function $g_y^{(2)}(x)$ of the moment sequence $\mu_n^{(2)}(y)$ can be expressed as the continued fraction $$\cfrac{1}{1-(2-y)x-
\cfrac{(1-y)x^2}{1-2x-
\cfrac{x^2}{1-2x-
\cfrac{x^2}{1-2x-\cdots}}}}.$$
\end{proposition}
\begin{proof} Let $G(x,y)$ be the continued fraction above. Then
$$G(x,y)=\frac{1}{1-(2-y)x-(1-y)x^2 G_1(x,y)},$$ where $G_1(x,y)$ is defined by the continued fraction
$$\frac{1}{1-2x-
\cfrac{x^2}{1-2x-
\cfrac{x^2}{1-2x-\cdots}}}.$$
It is well known that $G_1(x,y)=c(x)^2$. Thus
$$G(x,y)=\frac{1}{1-(2-y)x-(1-y)x^2 c(x)^2}=\frac{2}{1+y-2x+(1-y)\sqrt{1-4x}}.$$
A simple calculation now shows that $G(x,y)=g_y^{(2)}(x)$.
\end{proof}
\begin{corollary}
The Hankel transform $h_n^{(2)}(y)$ of the moments $\mu_n^{(2)}(y)$ is given by
$$h_n(y)=(1-y)^n.$$
\end{corollary}
Thus the coefficient array of the polynomial sequence $h_n^{(2)}(y)$ given by the Hankel transform of $\mu_n^{(2)}(y)$ is given by the signed Pascal's triangle $\left((-1)^n \binom{n}{k}\right)$ which begins
$$\left(
\begin{array}{ccccccc}
 1 & 0 & 0 & 0 & 0 & 0 & 0 \\
 1 & -1 & 0 & 0 & 0 & 0 & 0 \\
 1 & -2 & 1 & 0 & 0 & 0 & 0 \\
 1 & -3 & 3 & -1 & 0 & 0 & 0 \\
 1 & -4 & 6 & -4 & 1 & 0 & 0 \\
 1 & -5 & 10 & -10 & 5 & -1 & 0 \\
 1 & -6 & 15 & -20 & 15 & -6 & 1 \\
\end{array}
\right).$$
The coefficient array of the polynomial sequence $h_n^{(1)}(y)$ given by the Hankel transform of $\mu_n^{(1)}(y)$ begins
$$\left(
\begin{array}{cccccc}
 1 & 0 & 0 & 0 & 0 & 0 \\
 1 & -2 & 0 & 0 & 0 & 0 \\
 1 & -4 & 3 & 0 & 0 & 0 \\
 1 & -6 & 10 & -4 & 0 & 0 \\
 1 & -8 & 21 & -20 & 5 & 0 \\
 1 & -10 & 36 & -56 & 35 & -6 \\
\end{array}
\right),$$ with general term $(-1)^k \binom{2n+1-k}{2n+1-2k}$.
The reversal of this triangle, which begins
$$\left(
\begin{array}{cccccc}
 1 & 0 & 0 & 0 & 0 & 0 \\
 -2 & 1 & 0 & 0 & 0 & 0 \\
 3 & -4 & 1 & 0 & 0 & 0 \\
 -4 & 10 & -6 & 1 & 0 & 0 \\
 5 & -20 & 21 & -8 & 1 & 0 \\
 -6 & 35 & -56 & 36 & -10 & 1 \\
\end{array}
\right),$$ is the Riordan array $\left(\frac{1}{(1+x)^2}, \frac{x}{(1+x)^2}\right)$.

\section{A general result}
Using the same methods as above, we have the following more general result.
\begin{proposition} The coefficient array of the moments of the paramaterized orthogonal polynomials whose coefficient array is given by the Riordan array
$$\left(\frac{1+(2-a+y)x+(-a+b+1+y)x^2}{1+ax+bx^2}, \frac{x}{1+ax+bx^2}\right)$$ is given by the Riordan involution $(g(x), f(x))$ where we have
$$g(x)=\frac{2b}{1-a+2b+(a-1)(a-4b)x+(a-1)\sqrt{1-2ax+(a^2-4b)x^2}},$$ and
$$f(x)=\frac{\sqrt{1-2ax+(a^2-4b)x^2}+(a-2b)x-1}{1-a+2b+(a-1)(a-4b)x+(a-1)\sqrt{1-2ax+(a^2-4b)x^2}}.$$
\end{proposition}
In this case, the moments will have a generating function expressible as the continued fraction
$$\cfrac{1}{1-(-y+2a-2)x+
\cfrac{(y-a+1)x^2}{1-ax-
\cfrac{bx^2}{1-ax-
\cfrac{bx^2}{1-ax-\cdots}}}}.$$
The moments thus have the Hankel transform
$$h_n=(-1)^n (y-a+1)^n b^{\binom{n}{2}}.$$
In particular, the first column elements of the Riordan involution (the case $y=0$) have their generating function expressible as
$$\cfrac{1}{1-(2a-2)x-
\cfrac{(a-1)x^2}{1-ax-
\cfrac{bx^2}{1-ax-
\cfrac{bx^2}{1-ax-\cdots}}}},$$
with Hankel transform
$$h_n=(a-1)^n b^{\binom{n}{2}}.$$
\begin{example} We take the case of $a=3$ and $b=2$. Thus the coefficient matrix of the orthogonal polynomials in question is given by
$$\left(\frac{1+(y-1)x+yx^2}{1+3x+2x^2}, \frac{x}{1+3x+2x^2}\right).$$ The moments then have generating function $$\mu(x)=\frac{4}{2+y-(10-y)x-(y-2)\sqrt{1-6x+x^2}}.$$
This is then the bivariate generating function of the Riordan involution
$$\left(\frac{\sqrt{1-6x+x^2}+5x-1}{2x(1-6x)}, \frac{(2x-1)\sqrt{1-6x+x^2}-2x^2-5x+1}{4x(1-6x)}\right).$$
This matrix begins
$$\left(
\begin{array}{ccccccc}
 1 & 0 & 0 & 0 & 0 & 0 & 0 \\
 4 & -1 & 0 & 0 & 0 & 0 & 0 \\
 18 & -9 & 1 & 0 & 0 & 0 & 0 \\
 86 & -63 & 14 & -1 & 0 & 0 & 0 \\
 426 & -403 & 133 & -19 & 1 & 0 & 0 \\
 2162 & -2469 & 1070 & -228 & 24 & -1 & 0 \\
 11166 & -14769 & 7857 & -2212 & 348 & -29 & 1 \\
\end{array}
\right).$$
The first column of this involution matrix begins
$$1, 4, 18, 86, 426, 2162, 11166, 58438, 309042, 1648154, 8851206,\ldots$$ which is \seqnum{A225887}. This counts the number of Schr\"oder paths of semi-length $n$ for which the $(2,0)$-steps that are on the horizontal axis come in $3$ colors. This sequence has generating function given by the continued fraction
$$\cfrac{1}{1-4x-
\cfrac{2x^2}{1-3x-
\cfrac{2x^2}{1-3x-
\cfrac{2x^2}{1-3x-\cdots}}}}.$$
The row sums of the involution matrix begin
$$1, 3, 10, 36, 138, 558, 2362, 10398, 47326,\ldots.$$
These have their generating function given by the continued fraction
$$\cfrac{1}{1-3x-
\cfrac{x^2}{1-3x-
\cfrac{2x^2}{1-3x-
\cfrac{2x^2}{1-3x-\cdots}}}}.$$
The row sums of the matrix obtained by taking the absolute value of the entries of the involution matrix begin
$$1, 5, 28, 164, 982, 5954, 36382, 223466, 1377538,\ldots.$$
The generating function of this sequence can be expressed as
$$\cfrac{1}{1-5x-
\cfrac{3x^2}{1-3x-
\cfrac{2x^2}{1-3x-
\cfrac{2x^2}{1-3x-\cdots}}}}.$$
\end{example}
\begin{example} For the case $a=1$ and $b=2$ we obtain the Riordan involution
$$\left(1, \frac{\sqrt{1-2x-7x^2}-3x-1}{4}\right)$$ which begins
$$\left(
\begin{array}{ccccccc}
 1 & 0 & 0 & 0 & 0 & 0 & 0 \\
 0 & -1 & 0 & 0 & 0 & 0 & 0 \\
 0 & -1 & 1 & 0 & 0 & 0 & 0 \\
 0 & -1 & 2 & -1 & 0 & 0 & 0 \\
 0 & -3 & 3 & -3 & 1 & 0 & 0 \\
 0 & -7 & 8 & -6 & 4 & -1 & 0 \\
 0 & -21 & 21 & -16 & 10 & -5 & 1 \\
\end{array}
\right).$$
The row sums of this matrix begin
$$1, -1, 0, 0, -2, -2, -10, -26, -86,\ldots$$ with generating function given by
$$\cfrac{1}{1+x+
\cfrac{x^2}{1-x-
\cfrac{2x^2}{1-x-
\cfrac{2x^2}{1-x-\cdots}}}}.$$
This sequence has its Hankel transform given by
$$h_n =(-1)^n 2^{\binom{n}{2}}.$$
The row sums of the absolute value matrix begin
$$1, 1, 2, 4, 10, 26, 74, 218, 670,\ldots.$$
This sequence has a generating function expressible as the continued fraction
$$\cfrac{1}{1-x-
\cfrac{x^2}{1-x-
\cfrac{2x^2}{1-x-
\cfrac{2x^2}{1-x-\cdots}}}}.$$
Its Hankel transform is given by
$$h_n = 2^{\binom{n}{2}}.$$
\end{example}

\section{A factorization theorem}
We begin this section with an example that will motivate the general factorization result that will follow.
\begin{example}
We consider the Riordan array $\left(\frac{1}{(1+x)^2}, \frac{x}{(1+x)^2}\right)$ which begins
$$\left(
\begin{array}{ccccccc}
 1 & 0 & 0 & 0 & 0 & 0 & 0 \\
 -2 & 1 & 0 & 0 & 0 & 0 & 0 \\
 3 & -4 & 1 & 0 & 0 & 0 & 0 \\
 -4 & 10 & -6 & 1 & 0 & 0 & 0 \\
 5 & -20 & 21 & -8 & 1 & 0 & 0 \\
 -6 & 35 & -56 & 36 & -10 & 1 & 0 \\
 7 & -56 & 126 & -120 & 55 & -12 & 1 \\
\end{array}
\right).$$ The general element of this matrix is $(-1)^{n-k}\binom{n+k+1}{2k+1}$. The inverse of this matrix
is given by
$$\left(\frac{1}{(1+x)^2}, \frac{x}{(1+x)^2}\right)^{-1}=(c(x)^2, x c(x)^2),$$ which begins
$$\left(
\begin{array}{ccccccc}
 1 & 0 & 0 & 0 & 0 & 0 & 0 \\
 2 & 1 & 0 & 0 & 0 & 0 & 0 \\
 5 & 4 & 1 & 0 & 0 & 0 & 0 \\
 14 & 14 & 6 & 1 & 0 & 0 & 0 \\
 42 & 48 & 27 & 8 & 1 & 0 & 0 \\
 132 & 165 & 110 & 44 & 10 & 1 & 0 \\
 429 & 572 & 429 & 208 & 65 & 12 & 1 \\
\end{array}
\right).$$
We now form the Riordan array given by the following product.
$$\left(\frac{1}{(1+x)^2}, \frac{x}{(1+x)^2}\right) \cdot (c(-x)^2, -x c(-x)^2).$$
We obtain the Riordan array
$$\left(\frac{1+4x+x^2-(1+x)\sqrt{1+6x+x^2}}{2x^2},\frac{1+4x+x^2-(1+x)\sqrt{1+6x+x^2}}{2x}\right).$$
This matrix begins
$$\left(
\begin{array}{ccccccc}
 1 & 0 & 0 & 0 & 0 & 0 & 0 \\
 -4 & -1 & 0 & 0 & 0 & 0 & 0 \\
 16 & 8 & 1 & 0 & 0 & 0 & 0 \\
 -68 & -48 & -12 & -1 & 0 & 0 & 0 \\
 304 & 264 & 96 & 16 & 1 & 0 & 0 \\
 -1412 & -1408 & -652 & -160 & -20 & -1 & 0 \\
 6752 & 7432 & 4080 & 1296 & 240 & 24 & 1 \\
\end{array}
\right).$$
This matrix is a Riordan involution.
\end{example}
In a similar vein, we have the following result.
\begin{theorem} For a given Riordan array $(g(x), f(x))$, the product
$$(g(x), f(x))\cdot \left(\frac{1}{g(\bar{f}(-x))}, \bar{f}(-x)\right)$$ is a Riordan involution. \end{theorem}
\begin{proof}
We have that
$$(g(x), f(x))\cdot \left(\frac{1}{g(\bar{f}(-x))}, \bar{f}(-x)\right)=\left(\frac{g(x)}{g(\bar{f}(-f(x)))}, \bar{f}(-f(x))\right).$$
Then
\begin{align*}
\left(\frac{g(x)}{g(\bar{f}(-f(x)))}, \bar{f}(-f(x))\right)^2&=\left(\frac{g(x)}{g(\bar{f}(-f(x)))}, \bar{f}(-f(x))\right)\cdot \left(\frac{g(x)}{g(\bar{f}(-f(x)))}, \bar{f}(-f(x))\right)\\
&=\left(\frac{g(x)}{g(\bar{f}(-f))} \frac{g(\bar{f}(-f))}{g(\bar{f}(-f(\bar{f}(-f(x)))))}, \bar{f}(-f(\bar{f}(-f(x))))\right)\\
&=\left(\frac{g(x)}{g(\bar{f}(-f(\bar{f}(-f(x)))))}, \bar{f}(-(-f(x)))\right)\\
&=(1, x).\end{align*}
\end{proof}
Thus if we begin with the Riordan array
$$(g, f)=\left(\frac{1+ c x + d x^2}{1+a x + b x^2}, \frac{x}{1+ a x + b x^2}\right),$$ we are led to the Riordan involution given by $(G, F)$ where we have
\begin{scriptsize}
$$G(x)=\frac{b(1+cx+dx^2)(S-bx^2-2ax-1)}
{(d+(2ad-bc)x+bdx^2)S-b^2dx^4+b(bc-4ad)x^3+2(2a^2d-abc+b^2)x^2+(bc-4ad)x-d},$$
\end{scriptsize} and
$$F(x)=\frac{-1-2ax-bx^2+S}{2bx}=\frac{-x}{1+2ax+bx^2}c\left(\frac{bx^2}{(1+2ax+bx^2)^2}\right),$$
with
$$S=\sqrt{1+4ax+2(2a^2-b)x^2+4abx^3+b^2x^4}.$$
\begin{corollary}
The Riordan array
$$\left(1,\frac{-x}{1+2ax+bx^2}c\left(\frac{bx^2}{(1+2ax+bx^2)^2}\right)\right)$$ is a Riordan involution.
\end{corollary}
Similar results to this corollary may be found in \cite{Elliptic}.
\begin{example} We take the example of
$$(g, f)=\left(\frac{1+x+x^2}{1+2x+x^2}, \frac{x}{1+2x+x^2}\right).$$
We have
$$(g,f)^{-1}=\left(\frac{1}{1-x}, c(x)-1\right).$$
We form the product
$$\left(\frac{1+x+x^2}{1+2x+x^2}, \frac{x}{1+2x+x^2}\right)\cdot \left(\frac{1}{1+x}, c(-x)-1\right).$$
In terms of matrices, this begins
$$\left(
\begin{array}{cccccc}
 1 & 0 & 0 & 0 & 0 & 0 \\
 -1 & 1 & 0 & 0 & 0 & 0 \\
 2 & -3 & 1 & 0 & 0 & 0 \\
 -3 & 7 & -5 & 1 & 0 & 0 \\
 4 & -14 & 16 & -7 & 1 & 0 \\
 -5 & 25 & -41 & 29 & -9 & 1 \\
\end{array}
\right)\cdot \left(
\begin{array}{cccccc}
 1 & 0 & 0 & 0 & 0 & 0 \\
 -1 & -1 & 0 & 0 & 0 & 0 \\
 1 & 3 & 1 & 0 & 0 & 0 \\
 -1 & -8 & -5 & -1 & 0 & 0 \\
 1 & 22 & 19 & 7 & 1 & 0 \\
 -1 & -64 & -67 & -34 & -9 & -1 \\
\end{array}
\right),$$ with a product that begins
$$\left(
\begin{array}{cccccc}
 1 & 0 & 0 & 0 & 0 & 0 \\
 -2 & -1 & 0 & 0 & 0 & 0 \\
 6 & 6 & 1 & 0 & 0 & 0 \\
 -16 & -30 & -10 & -1 & 0 & 0 \\
 42 & 140 & 70 & 14 & 1 & 0 \\
 -110 & -642 & -424 & -126 & -18 & -1 \\
\end{array}
\right).$$
This is the Riordan involution
$$\left(\frac{1+x+x^2}{1+3x+x^2}, c\left(\frac{-x}{(1+x)^2}\right)-1\right).$$
We can write this as
$$\left(\frac{1+x+x^2}{1+3x+x^2}, \frac{-x}{1+4x+x^2}c\left(\frac{x^2}{(1+4x+x^2)^2}\right)\right).$$
In particular, we see that
$$\left(1, \frac{-x}{1+4x+x^2}c\left(\frac{x^2}{(1+4x+x^2)^2}\right)\right)$$ is a Riordan involution.
This matrix begins
$$\left(
\begin{array}{ccccccc}
 1 & 0 & 0 & 0 & 0 & 0 & 0 \\
 0 & -1 & 0 & 0 & 0 & 0 & 0 \\
 0 & 4 & 1 & 0 & 0 & 0 & 0 \\
 0 & -16 & -8 & -1 & 0 & 0 & 0 \\
 0 & 68 & 48 & 12 & 1 & 0 & 0 \\
 0 & -304 & -264 & -96 & -16 & -1 & 0 \\
 0 & 1412 & 1408 & 652 & 160 & 20 & 1 \\
\end{array}
\right).$$
The generating function $\frac{-x}{1+4x+x^2}c\left(\frac{x^2}{(1+4x+x^2)^2}\right)$ expands to give the sequence that begins
$$0,-1, 4, -16, 68, -304, 1412,\ldots.$$
The sequence $1, 4, 16, 68, 304, 1412,\ldots$ is \seqnum{A006319}$(n+1)$, which counts the number of peaks at level $1$ in all Schr\"oder paths of semi-length $n$. The row sums of the Riordan involution $\left(1, \frac{-x}{1+4x+x^2}c\left(\frac{x^2}{(1+4x+x^2)^2}\right)\right)$ begins
$$1, -1, 5, -25, 129, -681, 3653, -19825, 108545, -598417, 3317445, \ldots.$$ The Hankel transform of this sequence is given by
$$h_n = 2^{\binom{n+1}{2}} [x^n]\frac{1}{1-2x+2x^2}=2^{\binom{n+1}{2}}n![x^n] e^x(\cos(x)+\sin(x)).$$
\end{example}
Specialising to the case of the orthogonal polynomials defined by  $$(g(x), f(x))=\left(\frac{1}{1+ax + b x^2}, \frac{x}{1+a x + b x^2}\right),$$ we have the following corollary.
\begin{corollary} The orthogonal polynomial coefficient array 
$$(g(x), f(x))=\left(\frac{1}{1+ax + b x^2}, \frac{x}{1+a x + b x^2}\right)$$ defines the Riordan involution given by 
$$\left(\frac{1}{1+2ax + bx^2}c\left(\frac{bx^2}{(1+2ax+bx^2)^2}\right), \frac{-x}{1+2ax + bx^2}c\left(\frac{bx^2}{(1+2ax+bx^2)^2}\right)\right).$$ 
\end{corollary}
\begin{example} \textbf{The RNA involution}. We start with the Riordan array 
$$(g(x), f(x))=\left(\frac{1}{1-\frac{x}{2}+x^2}, \frac{x}{1-\frac{x}{2}+x^2}\right).$$ This matrix begins 
$$\left(
\begin{array}{cccccc}
 1 & 0 & 0 & 0 & 0 & 0 \\
 \frac{1}{2} & 1 & 0 & 0 & 0 & 0 \\
 -\frac{3}{4} & 1 & 1 & 0 & 0 & 0 \\
 -\frac{7}{8} & -\frac{5}{4} & \frac{3}{2} & 1 & 0 & 0 \\
 \frac{5}{16} & -\frac{5}{2} & -\frac{3}{2} & 2 & 1 & 0 \\
 \frac{33}{32} & \frac{5}{16} & -\frac{19}{4} & -\frac{3}{2} & \frac{5}{2} & 1 \\
\end{array}
\right).$$
The matrix $\left(\frac{1}{g(\bar{f}(-x))}, \bar{f}(-x)\right)$ then begins
$$\left(
\begin{array}{cccccc}
 1 & 0 & 0 & 0 & 0 & 0 \\
 \frac{1}{2} & -1 & 0 & 0 & 0 & 0 \\
 \frac{5}{4} & -1 & 1 & 0 & 0 & 0 \\
 \frac{13}{8} & -\frac{11}{4} & \frac{3}{2} & -1 & 0 & 0 \\
 \frac{57}{16} & -\frac{9}{2} & \frac{9}{2} & -2 & 1 & 0 \\
 \frac{201}{32} & -\frac{165}{16} & \frac{35}{4} & -\frac{13}{2} & \frac{5}{2} & -1 \\
\end{array}
\right).$$ 
Taking the product of these two matrices, we obtain the matrix that begins 
$$\left(
\begin{array}{ccccccc}
 1 & 0 & 0 & 0 & 0 & 0 & 0 \\
 1 & -1 & 0 & 0 & 0 & 0 & 0 \\
 1 & -2 & 1 & 0 & 0 & 0 & 0 \\
 2 & -3 & 3 & -1 & 0 & 0 & 0 \\
 4 & -6 & 6 & -4 & 1 & 0 & 0 \\
 8 & -13 & 13 & -10 & 5 & -1 & 0 \\
 17 & -28 & 30 & -24 & 15 & -6 & 1 \\
\end{array}
\right),$$ which is the RNA involution \cite{Cameron}, defined by 
$$\left(\frac{1}{1-x+x^2}c\left(\frac{x^2}{(1-x+x^2)^2}\right), \frac{-x}{1-x+x^2}c\left(\frac{x^2}{(1-x+x^2)^2}\right)\right).$$ 
\end{example}

\section{Further results}
The matrix $M$ with  $M^{-1}=\left(\frac{1}{(1+x)^2}, \frac{x}{(1+x)^2}\right)$ is the coefficient array of a family of orthogonal polynomials. This is evidenced by the fact that the production matrix of $M$ is given by the matrix that begins
$$\left(
\begin{array}{ccccccc}
 2 & 1 & 0 & 0 & 0 & 0 & 0 \\
 1 & 2 & 1 & 0 & 0 & 0 & 0 \\
 0 & 1 & 2 & 1 & 0 & 0 & 0 \\
 0 & 0 & 1 & 2 & 1 & 0 & 0 \\
 0 & 0 & 0 & 1 & 2 & 1 & 0 \\
 0 & 0 & 0 & 0 & 1 & 2 & 1 \\
 0 & 0 & 0 & 0 & 0 & 1 & 2 \\
\end{array}
\right).$$
Indeed, it is the three-diagonal form of this matrix that tells us that we are working with orthogonal polynomials. In like manner, the matrix
$$M=\left(\frac{1}{(1+x)^3}, \frac{x}{(1+x)^3}\right)^{-1}$$ has a production matrix that begins
$$\left(
\begin{array}{ccccccc}
 3 & 1 & 0 & 0 & 0 & 0 & 0 \\
 3 & 3 & 1 & 0 & 0 & 0 & 0 \\
 1 & 3 & 3 & 1 & 0 & 0 & 0 \\
 0 & 1 & 3 & 3 & 1 & 0 & 0 \\
 0 & 0 & 1 & 3 & 3 & 1 & 0 \\
 0 & 0 & 0 & 1 & 3 & 3 & 1 \\
 0 & 0 & 0 & 0 & 1 & 3 & 3 \\
\end{array}
\right).$$
The four-diagonal form of this matrix tells us the $M^{-1}$ is the coefficient array of a family of $2$-orthogonal polynomials, where a $1$-orthogonal family of polynomials is a usual orthogonal family of polynomials.

We begin our discussion of the general case by looking at the case of the ternary numbers $T_n=\frac{1}{2n+1}\binom{3n}{n}$ \seqnum{A001764}. We denote by $t(x)$ the generating function of the ternary numbers, which satisfies $t(x)= 1+ x t(x)^3$.  We have
$$t(x)=\frac{2}{\sqrt{3x}}\sin\left(\frac{1}{3}\sin^{-1}\left(\frac{\sqrt{27x}}{2}\right)\right)$$
and
$$t(x)=1+\text{Rev}\left(\frac{x}{(1+x)^3}\right).$$
\begin{example}
We consider the elements in the first column of the inverse Riordan array
$$\left(\frac{1+xy(1+x)^2}{(1+x)^3}, \frac{x}{(1+x)^3}\right)^{-1}.$$
These are polynomials in $y$ which begin
$$1, 3 - y, y^2 - 8y + 12, - y^3 + 13y^2 - 52y + 55, y^4 - 18y^3 + 117y^2 - 320y + 273,\ldots.$$
They have a coefficient array which begins
$$\left(
\begin{array}{cccccc}
 1 & 0 & 0 & 0 & 0 & 0 \\
 3 & -1 & 0 & 0 & 0 & 0 \\
 12 & -8 & 1 & 0 & 0 & 0 \\
 55 & -52 & 13 & -1 & 0 & 0 \\
 273 & -320 & 117 & -18 & 1 & 0 \\
 1428 & -1938 & 910 & -207 & 23 & -1 \\
\end{array}
\right).$$
This is the Riordan array $(t(x)^3, -xt(x)^5)$ which by \cite{Phulara} is a Riordan involution.
The Hankel transform of the (moment) polynomials is of combinatorial interest. This Hankel transform begins
$$1, 3 - 2y, 11y^2 - 34y + 26, - 170y^3 + 804y^2 - 1254y + 646, \ldots$$
which for $y=0,1$ gives, respectively, the sequences  beginning
$$ 1, 3, 26, 646, 45885, 9304650, 5382618660, 8878734657276,\ldots,$$ and
$$1, 1, 3, 26, 646, 45885, 9304650, 5382618660, 8878734657276,\ldots$$
The second sequence is \seqnum{A005156}, the number of alternating sign $(2n+1) \times (2n+1)$ matrices symmetric about the vertical axis. The values of the (moment) polynomials for $y=0,1$ are given by
$$1, 3, 12, 55, 273, 1428, 7752, 43263, 246675, 1430715, 8414640,\ldots$$ which is $T_{n+1}$ and
$$1, 2, 5, 15, 53, 215, 971, 4745, 24540, 132235, 734572,\ldots.$$
This latter sequence is the INVERT$(1)$ transform of \seqnum{A098746}, which gives the number of permutations of $[n]$ which avoid $4231$ and $42513$.

The coefficient array of the polynomial sequence given by the Hankel transform begins
$$\left(
\begin{array}{cccccc}
 1 & 0 & 0 & 0 & 0 & 0 \\
 3 & -2 & 0 & 0 & 0 & 0 \\
 26 & -34 & 11 & 0 & 0 & 0 \\
 646 & -1254 & 804 & -170 & 0 & 0 \\
 45885 & -117990 & 112860 & -47538 & 7429 & 0 \\
 9304650 & -29774880 & 37838910 & -23849850 & 7447515 & -920460 \\
\end{array}
\right).$$
We note that the right border in absolute value begins
$$1, 2, 11, 170, 7429, 920460, 323801820, 323674802088,\ldots,$$ which is \seqnum{A051255}$(n+1)$. The sequence
\seqnum{A051255} counts the number of cyclically symmetric transpose complement plane partitions in a $(2n)\times(2n)\times(2n)$ box. This is the Hankel transform of the ternary numbers $\frac{1}{2n+1}\binom{3n}{n}$.

In like fashion, the moments of the matrix
$$\left(\frac{1+xy(1+x)^2}{(1+x)^2}, \frac{x}{(1+x)^3}\right)$$ have coefficient matrix given by the Riordan involution $(t(x)^2, -xt(x)^5)$, the moments of the matrix
$$\left(\frac{1+xy(1+x)^2}{(1+x)}, \frac{x}{(1+x)^3}\right)$$ have coefficient matrix given by the Riordan involution $(t(x), -xt(x)^5)$, and the moments of the matrix
$$\left(1+xy(1+x)^2, \frac{x}{(1+x)^3}\right)$$ have their coefficient matrix given by the Riordan involution $(1, -xt(x)^5)$.
\end{example}
The methods of \cite{Phulara} now allow us to give the general result.
\begin{theorem} The coefficient matrix of the moments of the parameterized Riordan array
$$\left(\frac{1+xy(1+x)^{k-1}}{(1+x)^m}, \frac{x}{(1+x)^k}\right)$$
for $m=0\ldots k$ is the Riordan involution $$\left(g(x)^m, -x g(x)^{2k-1}\right),$$
where $g(x)=1+x g(x)^k$.
\end{theorem}
The moments in question are then the moments of the corresponding  $(k-1)$-orthogonal polynomials \cite{DO}.
\section{Conclusions} We have seen that parameterized Riordan arrays of the form
$$\left(\frac{1+xy(1+x)^{k-1}}{(1+x)^m}, \frac{x}{(1+x)^k}\right)$$ are closely related to certain Riordan involutions. These involutions arise as the coefficient arrays (expansions in the parameter $y$) of the polynomial sequences (in $y$) that arise as the first columns of the inverse matrix. In the case of $k=2$, the polynomials concerned are generalized Chebyshev polynomials. In the cases studied above, the  polynomial sequences have interesting Hankel transforms.

Moreover, we have seen that for every Riordan array $(g(x), f(x))$, there is a Riordan involution given by the product 
$$ (g(x), f(x))\cdot \left(\frac{1}{g(\bar{f}(-x))}, \bar{f}(-x)\right).$$

\section{Appendix - a look at the case of the Chebyshev polynomials of the first kind}
For the sake of completeness, we present a small discussion of the Chebyshev polynomials of the first kind $T_n(x)$ and related moments. While the Chebyshev polynomials of the second kind have a Riordan array as their coefficient array - namely, the Riordan array $\left(\frac{1}{1+x^2}, \frac{2x}{1+x^2}\right)$, the coefficient array for the Chebyshev polynomials of the first kind is not given by a Riordan array, but by an almost-Riordan array of the first kind \cite{Almost}. This means that the coefficient array is given by a Riordan array, with a column prepended, which in this case means that the overall matrix is not a Riordan array.

We have
$$\sum_{n=0}^{\infty}T_n(x)t^n = \frac{1- x t}{1-2x t+t^2}.$$
The sequence $T_n(x)$ begins
$$1, x, 2x^2 - 1, 4x^3 - 3x, 8x^4 - 8x^2 + 1, 16x^5 - 20x^3 + 5x,\ldots,$$ with a coefficient array
that begins
$$\left(
\begin{array}{ccccccc}
 1 & 0 & 0 & 0 & 0 & 0 & 0 \\
 0 & 1 & 0 & 0 & 0 & 0 & 0 \\
 -1 & 0 & 2 & 0 & 0 & 0 & 0 \\
 0 & -3 & 0 & 4 & 0 & 0 & 0 \\
 1 & 0 & -8 & 0 & 8 & 0 & 0 \\
 0 & 5 & 0 & -20 & 0 & 16 & 0 \\
 -1 & 0 & 18 & 0 & -48 & 0 & 32 \\
\end{array}
\right).$$
The first column has generating function $\frac{1}{1+x^2}$, while the embedded Riordan array that begins
$$\left(
\begin{array}{cccccc}
 1 & 0 & 0 & 0 & 0 & 0 \\
 0 & 2 & 0 & 0 & 0 & 0 \\
 -3 & 0 & 4 & 0 & 0 & 0 \\
 0 & -8 & 0 & 8 & 0 & 0 \\
 5 & 0 & -20 & 0 & 16 & 0 \\
 0 & 18 & 0 & -48 & 0 & 32 \\
\end{array}
\right)$$ is the Riordan array
$$\left(\frac{1-x^2}{(1+x^2)^2}, \frac{2x}{1+x^2}\right).$$
As an almost Riordan array of the first order, we write the coefficient array as
$$\left(\frac{1}{1+x^2} | \frac{1-x^2}{(1+x^2)^2}, \frac{2x}{1+x^2}\right).$$
As before, we shall use the term ``moments'' for the first column elements of the inverses of the matrices that we shall present.

We consider the parameterized almost Riordan array, based on the generalized Chebyshev polynomials of the first kind, given by
$$\left(\frac{1+yx+yx^2}{1+x^2} | \frac{1+yx+yx^2}{(1+x^2)^2}, \frac{x}{1+x^2}\right).$$
This almost Riordan array begins
$$\left(
\begin{array}{ccccccc}
 1 & 0 & 0 & 0 & 0 & 0 & 0 \\
 y & 1 & 0 & 0 & 0 & 0 & 0 \\
 y-1 & y & 1 & 0 & 0 & 0 & 0 \\
 -y & y-2 & y & 1 & 0 & 0 & 0 \\
 1-y & -2 y & y-3 & y & 1 & 0 & 0 \\
 y & 3-2 y & -3 y & y-4 & y & 1 & 0 \\
 y-1 & 3 y & 3\cdot (2-y) & -4 y & y-5 & y & 1 \\
\end{array}
\right),$$ with an inverse that begins
$$\left(
\begin{array}{ccccc}
 1 & 0 & 0 & 0 & 0 \\
 -y & 1 & 0 & 0 & 0 \\
 y^2-y+1 & -y & 1 & 0 & 0 \\
 -y \left(y^2-2 y+2\right) & y^2-y+2 & -y & 1 & 0 \\
 y^4-3 y^3+4 y^2-3 y+2 & -y \left(y^2-2 y+3\right) & y^2-y+3 & -y & 1 \\
\end{array}
\right).$$
The moment sequence we seek then begins
$$1, -y, y^2 - y + 1, - y(y^2 - 2y + 2), y^4 - 3y^3 + 4y^2 - 3y + 2, - y(y^4 - 4y^3 + 7y^2 - 8y + 5),\ldots.$$
This moment sequence thus has a coefficient array that begins
$$\left(
\begin{array}{cccccccc}
 1 & 0 & 0 & 0 & 0 & 0 & 0 & 0 \\
 0 & -1 & 0 & 0 & 0 & 0 & 0 & 0 \\
 1 & -1 & 1 & 0 & 0 & 0 & 0 & 0 \\
 0 & -2 & 2 & -1 & 0 & 0 & 0 & 0 \\
 2 & -3 & 4 & -3 & 1 & 0 & 0 & 0 \\
 0 & -5 & 8 & -7 & 4 & -1 & 0 & 0 \\
 5 & -9 & 14 & -16 & 11 & -5 & 1 & 0 \\
 0 & -14 & 28 & -32 & 28 & -16 & 6 & -1 \\
\end{array}
\right).$$
This is the Riordan array
$$\left(c(x^2), \frac{(1+x)\sqrt{1-4x^2}+2x^2-x-1}{2x^2}\right)=\left(c(x^2), (1+x)c(x^2)-1\right).$$
In this case, we do not get an involution. Indeed, the inverse of this matrix is given by
$$\left(\frac{\sqrt{1-4x}-2x+3}{2(x^2-2x+2)}, -\frac{\sqrt{1-4x}(x-1)+x+1}{2(x^2-2x+2)}\right).$$
The generating function of the  moment sequence may be expressed as the continued fraction
$$\cfrac{1}{1+yx-
\cfrac{(1-y)x^2}{1-
\cfrac{x^2}{1-
\cfrac{x^2}{1-\cdots}}}},$$ which is equal to
$$\frac{1}{1+yx-(1-y)x^2 c(x^2)}.$$
We deduce from the continued fraction that the Hankel transform of the moment sequence is given by
$$h_n = (1-y)^n.$$
We note that the sequence with generating function $\frac{\sqrt{1-4x}-2x+3}{2(x^2-2x+2)}$ has an interesting property. This sequence begins
$$1, 0, -1, -2, -4, -10, -29, -90, -290, -960, -3246,\ldots.$$ This sequence, and the sequence \seqnum{A182486}, which begins
$$1,1, 0, -1, -2, -4, -10, -29, -90, -290, -960, -3246,\ldots$$ both have the Hankel transform $(-1)^n$, and this pair is the only pair of sequence and once-shifted sequence that has this property (Michael Somos).

The generating function $\frac{\sqrt{1-4x}(x-1)+x+1}{2(x^2-2x+2)}=\frac{x}{1-(1-x)xc(x)}$ expands to give the sequence \seqnum{A035929}.

\bigskip
\hrule

\noindent 2010 {\it Mathematics Subject Classification}: Primary
15B36; Secondary 33C45, 11B83, 11C20, 05A15.
\noindent \emph{Keywords:} Riordan array, involution, orthogonal polynomial, Chebyshev polynomials, moment sequence, Hankel transform.

\bigskip
\hrule
\bigskip
\noindent (Concerned with sequences
\seqnum{A000045},
\seqnum{A000108},
\seqnum{A001764},
\seqnum{A005156},
\seqnum{A007318},
\seqnum{A033184},
\seqnum{A035929},
\seqnum{A039598},
\seqnum{A051255},
\seqnum{A081696},
\seqnum{A098746},
\seqnum{A106566},
\seqnum{A107842},
\seqnum{A109262},
\seqnum{A109267},
\seqnum{A128899},
\seqnum{A182486}, and
\seqnum{A225887}.)

\end{document}